\documentclass[12pt]{amsproc}
\usepackage{hyperref}
\usepackage{geometry}
\usepackage{graphicx}
\usepackage{amsthm}
\usepackage{amssymb}

\theoremstyle{plain}
\newtheorem{thm}{Theorem}[section]

\newtheorem{lem}[thm]{Lemma}
\newtheorem{prop}[thm]{Proposition}
\newtheorem{defn}[thm]{Definition}

\newtheorem{rem}[thm]{Remark}

\begin{document}

\title{On Generalizations of Graded $2$-absorbing and Graded $2$-absorbing primary submodules}

\author{Azzh Saad \textsc{Alshehry}}
\address{Department of Mathematical Sciences, Faculty of Sciences, Princess Nourah Bint Abdulrahman University, P.O. Box 84428, Riyadh 11671, Saudi Arabia.}
\email{asalshihry@pnu.edu.sa}

\author{Rashid \textsc{Abu-Dawwas}}
\address{Department of Mathematics, Yarmouk University, Irbid, Jordan}
\email{rrashid@yu.edu.jo}



\begin{abstract}
Let $R$ be a graded commutative ring with non-zero unity $1$ and $M$ be a graded unitary $R$-module. In this article, we introduce the concepts of graded $\phi$-$2$-absorbing and graded $\phi$-$2$-absorbing primary submodules as generalizations of the concepts of graded $2$-absorbing and graded $2$-absorbing primary submodules. Let $GS(M)$ be the set of all graded $R$-submodules of $M$ and $\phi: GS(M)\rightarrow GS(M)\bigcup\{\emptyset\}$ be a function. A proper graded $R$-submodule $K$ of $M$ is said to be a graded $\phi$-$2$-absorbing $R$-submodule of $M$ if whenever $x, y$ are homogeneous elements of $R$ and $m$ is a homogeneous element of $M$ with $xym\in K-\phi(K)$, then $xm\in K$ or $ym\in K$ or $xy\in (K :_{R} M)$, and $K$ is said to be a graded $\phi$-$2$-absorbing primary $R$-submodule of $M$ if whenever $x, y$ are homogeneous elements of $R$ and $m$ is a homogeneous element of $M$ with $xym\in K-\phi(K)$, then $xm$ or $ym$ is in the graded radical of $K$ or $xy\in (K:_{R}M)$. We investigate several properties of these new types of graded submodules.
\end{abstract}

\maketitle

\section{Introduction}

Throughout this article, $G$ will be a group with identity $e$ and $R$ a commutative ring with nonzero unity $1$. Then $R$ is said to be $G$-graded if $R=\displaystyle\bigoplus_{g\in G} R_{g}$ with $R_{g}R_{h}\subseteq R_{gh}$ for all $g, h\in G$ where $R_{g}$ is an additive subgroup of $R$ for all $g\in G$. The elements of $R_{g}$ are called homogeneous of degree $g$. If $x\in R$, then $x$ can be written uniquely as $\displaystyle\sum_{g\in G}x_{g}$, where $x_{g}$ is the component of $x$ in $R_{g}$. It is known that $R_e$ is a subring of $R$ and $1\in R_e$. The set of all homogeneous elements of $R$ is $h(R)=\displaystyle\bigcup_{g\in G}R_{g}$. Assume that $M$ is a left unitary $R$-module. Then $M$ is said to be $G$-graded if $M=\displaystyle\bigoplus_{g\in G}M_{g}$ with $R_{g}M_{h}\subseteq M_{gh}$ for all $g,h\in G$ where $M_{g}$ is an additive subgroup of $M$ for all $g\in G$. The elements of $M_{g}$ are called homogeneous of degree $g$. It is clear that $M_{g}$ is an $R_{e}$-submodule of $M$ for all $g\in G$. If $x\in M$, then $x$ can be written uniquely as $\displaystyle\sum_{g\in G}x_{g}$, where $x_{g}$ is the component of $x$ in $M_{g}$. The set of all homogeneous elements of $M$ is $h(M)=\displaystyle\bigcup_{g\in G}M_{g}$. Let $K$ be an $R$-submodule of a graded $R$-module $M$. Then $K$ is said to be graded $R$-submodule if $K=\displaystyle\bigoplus_{g\in G}(K\cap M_{g})$, i.e., for $x\in K$, $x=\displaystyle\sum_{g\in G}x_{g}$ where $x_{g}\in K$ for all $g\in G$. An $R$-submodule of a graded $R$-module need not be graded. For more details and terminology, see \cite{Hazart, Nastasescue}.

\begin{lem}\label{1}(\cite{Farzalipour}, Lemma 2.1) Let $R$ be a graded ring and $M$ be a graded $R$-module.

\begin{enumerate}

\item If $I$ and $J$ are graded ideals of $R$, then $I+J$ and $I\bigcap J$ are graded ideals of $R$.

\item If $N$ and $K$ are graded $R$-submodules of $M$, then $N+K$ and $N\bigcap K$ are graded $R$-submodules of $M$.

\item If $N$ is a graded $R$-submodule of $M$, $r\in h(R)$, $x\in h(M)$ and $I$ is a graded ideal of $R$, then $Rx$, $IN$ and $rN$ are graded $R$-submodules of $M$. Moreover, $(N:_{R}M)=\left\{r\in R:rM\subseteq N\right\}$ is a graded ideal of $R$.
\end{enumerate}
\end{lem}

In particular, $Ann_{R}(M)=(0:_{R}M)$ is a graded ideal of $R$. Let $I$ be a proper graded ideal of $R$. Then the graded radical of $I$ is $Grad(I)$, and is defined to be the set of all $r\in R$ such that for each $g\in G$, there exists a positive integer $n_{g}$ for which $r_{g}^{n_{g}}\in I$. One can see that if $r\in h(R)$, then $r\in Grad(I)$ if and only if $r^{n}\in I$ for some positive integer $n$. In fact, $Grad(I)$ is a graded ideal of $R$, see \cite{Refai Hailat}. Let $N$ be a graded $R$-submodule of $M$. Then the graded radical of $N$ is denoted by $Grad_{M}(N)$ and it is defined to be the intersection of all graded prime submodules of $M$ containing $N$. If there is no graded prime submodule containing $N$, then we take $Grad_{M}(N)=M$.

 A graded prime (resp. graded primary) $R$-submodule is a proper graded $R$-submodule $K$ of $M$ with the property that for $r\in h(R)$ and $m\in h(M)$ such that $rm\in K$ implies that $m\in K$ or $r\in (K :_{R} M)$ (resp. $m\in K$ or $r\in Grad((K :_{R} M)))$. As graded prime ideals (submodules) have an important role in graded ring (module) theory, several authors generalized these concepts in different ways, see (\cite{Dawwas Shashan, Zoubi Dawwas 2, Zoubi Azaizeh, Ece Yetkin Celikel, Oral}). Graded weakly prime submodules have been introduced by Atani in \cite{Atani}. A proper graded $R$-submodule $K$ of $M$ is said to be a graded weakly prime $R$-submodule of $M$ if whenever $r\in h(R)$ and $m\in h(M)$ such that $0\neq rm\in K$, then $m\in K$ or $r\in (K :_{R} M)$. The concept of graded $\phi$-prime submodules has been introduced in \cite{Dawwas Alshehry}. Let $GS(M)$ be the set of all graded $R$-submodules of $M$ and $\phi: GS(M)\rightarrow GS(M)\bigcup\{\emptyset\}$ be a function. A proper graded $R$-submodule $K$ of $M$ is said to be a graded $\phi$-prime $R$-submodule of $M$ if whenever $r\in h(R)$ and $m\in h(M)$ such that $rm\in K-\phi(K)$, then $m\in K$ or $r\in (K:_{R}M)$.

The concept of graded $2$-absorbing ideals (resp. graded weakly $2$-absorbing ideals) is introduced in \cite{Zoubi Dawwas Ceken} as a different generalization of graded prime ideals (resp. graded weakly prime ideals). A proper graded ideal $I$ of $R$ is a graded $2$-absorbing ideal (resp. graded weakly $2$-absorbing ideal) of $R$ if whenever $x, y, z\in h(R)$ and $xyz\in I$ (resp. $0\neq xyz\in I$), then $xy\in I$ or $xz\in I$ or $yz\in I$. Then introducing graded $2$-absorbing submodules (resp. graded weakly $2$-absorbing submodules) in \cite{Zoubi Dawwas 2} generalized the concept of graded $2$-absorbing ideals (resp. graded weakly $2$-absorbing ideals) to graded submodules as following: A proper graded $R$-submodule $K$ of $M$ is said to be a graded $2$-absorbing $R$-submodule (resp. graded weakly $2$-absorbing $R$-submodule) of $M$ if whenever $x, y\in h(R)$ and $m\in h(M)$ with $xym\in K$ (resp. $0\neq xym\in K$), then $xy\in (K :_{R} M)$ or $xm\in K$ or $ym\in K$.

Al-Zoubi and Sharafat in \cite{Zoubi Sharafat} introduced the concept of graded $2$-absorbing primary ideals, where a proper graded ideal $I$ of $R$ is called graded $2$-absorbing primary if whenever $x, y, z\in h(R)$ with $xyz\in I$, then $xy\in I$ or $xz\in Grad(I)$ or $yz\in Grad(I)$. The concept of graded $2$-absorbing primary submodules is studied in \cite{Ece Yetkin Celikel} as a generalization of graded $2$-absorbing primary ideals. A proper graded $R$-submodule $K$ of $M$ is said to be a graded $2$-absorbing primary $R$-submodule (resp. graded weakly $2$-absorbing primary $R$-submodule) of $M$ if
whenever $xy\in h(R)$ and $m\in h(M)$ with $xym\in K$ (resp. $0\neq xym\in K$), then $xy\in (K :_{R} M)$ or $xm\in Grad_{M}(K)$ or $ym\in Grad_{M}(K)$.

A graded $R$-module $M$ is said to be graded multiplication if for every graded $R$-submodule $N$ of $M$, $N=IM$ for some graded deal $I$ of $R$. In this case, it is known that $I=(N:_{R}M)$. Graded multiplication modules were firstly introduced and studied by Escoriza and Torrecillas in \cite{Escoriza}, and further results were obtained by several authors, see for example \cite{Khaksari}. Let $N$ and $K$ be graded $R$-submodules of a graded multiplication $R$-module $M$ with $N = IM$ and $K = JM$ for some graded ideals $I$ and $J$ of $R$. The product of $N$ and $K$ is denoted by $NK$ is defined by $NK =IJM$. Then the product of $N$ and $K$ is independent of presentations of $N$ and $K$. In fact, as $IJ$ is a graded ideal of $R$ (see \cite{Nastasescue}), $NK$ is a graded $R$-submodule of $M$ and $NK\subseteq N\bigcap K$. Moreover, for $x, y\in h(M)$, by $xy$, we mean the product of $Rx$ and $Ry$. Also, it is shown in (\cite{Oral}, Theorem 9) that if $N$ is a proper graded $R$-submodule of a graded multiplication $R$-module $M$, then $Grad_{M}(N) =Grad((N :_{R} M))M$.

In this article, our aim is to extend the concept of graded $2$-absorbing submodules to graded $\phi$-$2$-absorbing submodules in completely different way from \cite{Dawwas Alshehry}, and also to extend graded $2$-absorbing primary submodules to graded $\phi$-$2$-absorbing primary submodules. Our study is inspired from \cite{Hojjat}.

\section{Graded $\phi$-$2$-Absorbing and Graded $\phi$-$2$-Absorbing Primary Submodules}

In this section, we introduce and study the concepts of graded $\phi$-$2$-absorbing and graded $\phi$-$2$-absorbing primary submodules.

\begin{defn}Let $M$ be a $G$-graded $R$-module and $\phi: GS(M)\rightarrow GS(M)\bigcup\{\emptyset\}$ be a function.
\begin{enumerate}
\item A proper graded $R$-submodule $K$ of $M$ is said to be a graded $\phi$-primary $R$-submodule of $M$ if whenever $r\in h(R)$ and $m\in h(M)$ with $rm\in K-\phi(K)$, then $m\in K$ or $r\in Grad((K :_{R} M))$.

\item Let $K$ be a graded $R$-submodule of $M$ and $g\in G$ such that $K_{g}\neq M_{g}$. Then $K$ is said to be a $g$-$\phi$-primary $R$-submodule of $M$ if whenever $r\in R_{e}$ and $m\in M_{g}$ with $rm\in K-\phi(K)$, then $m\in K$ or $r\in Grad((K :_{R} M))$.

\item A proper graded $R$-submodule $K$ of $M$ is said to be a graded $\phi$-$2$-absorbing $R$-submodule of $M$ if whenever $x, y\in h(R)$ and $m\in h(M)$ with $xym\in K-\phi(K)$, then $xm\in K$ or $ym\in K$ or $xy\in (K :_{R} M)$.

\item Let $K$ be a graded $R$-submodule of $M$ and $g\in G$ such that $K_{g}\neq M_{g}$. Then $K$ is said to be a $g$-$\phi$-$2$-absorbing $R$-submodule of $M$ if whenever $x, y\in R_{e}$ and $m\in M_{g}$ with $xym\in K-\phi(K)$, then $xm\in K$ or $ym\in K$ or $xy\in (K :_{R} M)$.

\item A proper graded $R$-submodule $K$ of $M$ is said to be a graded $\phi$-$2$-absorbing primary $R$-submodule of $M$ if whenever $x, y\in h(R)$ and $m\in h(M)$ with $xym\in K-\phi(K)$, then $xm\in Grad_{M}(K)$ or $ym\in Grad_{M}(K)$ or $xy\in (K:_{R}M)$.

\item Let $K$ be a graded $R$-submodule of $M$ and $g\in G$ such that $K_{g}\neq M_{g}$. Then $K$ is said to be a $g$-$\phi$-$2$-absorbing primary $R$-submodule of $M$ if whenever $x, y\in R_{e}$ and $m\in M_{g}$ with $xym\in K-\phi(K)$, then $xm\in Grad_{M}(K)$ or $ym\in Grad_{M}(K)$ or $xy\in (K:_{R}M)$.
\end{enumerate}
\end{defn}

\begin{rem}\label{Lemma 2.2}
\begin{enumerate}
\item Let $K$ be a graded $\phi$-primary $R$-submodule of a graded multiplication $R$-module $M$. Then

\begin{center}
$\phi_{\emptyset}(K)=\emptyset$\hspace{0.5cm} graded primary submodule,

\vspace{0.25cm}

$\phi_{0}(K)=\{0\}$\hspace{0.5cm} graded weakly primary submodule,

\vspace{0.25cm}

$\phi_{2}(K)=K^{2}$\hspace{0.5cm} graded almost primary submodule,

\vspace{0.25cm}

$\phi_{n}(K)=K^{n}$\hspace{0.5cm} graded $n$-almost primary submodule, and

\vspace{0.25cm}

$\phi_{\omega}(K)=\displaystyle\bigcap_{n=1}^{\infty}K^{n}$\hspace{0.5cm} graded $\omega$-primary submodule.
\end{center}

\item Let $K$ be a graded $\phi$-$2$-absorbing (resp. graded $\phi$-$2$-absorbing primary) $R$-submodule of a graded multiplication $R$-module $M$. Then

\begin{center}
$\phi_{\emptyset}(K)=\emptyset$\hspace{0.5cm} graded $2$-absorbing (resp. graded $2$-absorbing primary) submodule,

\vspace{0.25cm}

$\phi_{0}(K)=\{0\}$\hspace{0.5cm} graded weakly $2$-absorbing (resp. graded weakly $2$-absorbing primary) submodule,

\vspace{0.25cm}

$\phi_{2}(K)=K^{2}$\hspace{0.5cm} graded almost $2$-absorbing (resp. graded almost $2$-absorbing primary) submodule,

\vspace{0.25cm}

$\phi_{n}(K)=K^{n}$\hspace{0.5cm} graded $n$-almost $2$-absorbing (resp. graded $n$-almost $2$-absorbing primary) submodule, and

\vspace{0.25cm}

$\phi_{\omega}(K)=\displaystyle\bigcap_{n=1}^{\infty}K^{n}$\hspace{0.5cm} graded $\omega$-$2$-absorbing (resp. graded $\omega$-$2$-absorbing primary) submodule.
\end{center}

\item For functions $\phi, \varphi:GS(M)\rightarrow GS(M)\bigcup\{\emptyset\}$, we write $\phi\leq\varphi$ if $\phi(K)\subseteq\varphi(K)$ for all $K\in GS(M)$. Thus clearly we have the following order:
\begin{center}
$\phi_{\emptyset}\leq\phi_{0}\leq\phi_{\omega}\leq...\leq\phi_{n+1}\leq\phi_{n}\leq...\leq\phi_{2}\leq\phi_{1}$.
\end{center}

\item If $\phi\leq\varphi$, then every graded $\phi$-$2$-absorbing (resp. graded $\phi$-$2$-absorbing primary) $R$-submodule is graded $\varphi$-$2$-absorbing (resp. graded $\varphi$-$2$-absorbing primary).

\item Since $K-\phi(K) = K-(K\bigcap\phi(K))$ for any graded $R$-submodule $K$ of $M$, without loss of generality, throughout this article, we assume that $\phi(K)\subseteq K$.
\end{enumerate}
\end{rem}

\begin{thm}\label{Theorem 2.3} Let $M$ be a graded $R$-module and $K$ be a proper graded $R$-submodule of $M$. Then the followings hold:
\begin{enumerate}
\item $K$ is a graded $\phi$-prime $R$-submodule of $M$ $\Rightarrow$ $K$ is a graded $\phi$-$2$-absorbing $R$-submodule of $M$ $\Rightarrow$ $K$ is a graded $\phi$-$2$-absorbing primary $R$-submodule of $M$.

\item If $M$ is a graded multiplication $R$-module and $K$ is a graded $\phi$-primary $R$-submodule of $M$, then $K$ is a graded $\phi$-$2$-absorbing primary $R$-submodule of $M$.

\item For graded multiplication $R$-module $M$, $K$ is a graded $2$-absorbing $R$-submodule of $M$ $\Rightarrow$ $K$ is a graded weakly $2$-absorbing $R$-submodule of $M$ $\Rightarrow$ $K$ is a graded $\omega$-$2$-absorbing $R$-submodule of $M$ $\Rightarrow$ $K$ is a graded $(n+1)$-almost $2$-absorbing $R$-submodule of $M$ $\Rightarrow$ $K$ is a graded $n$-almost $2$-absorbing $R$-submodule of $M$ for all $n\geq2$ $\Rightarrow$ $K$ is a graded almost $2$-absorbing $R$-submodule of $M$.

\item For graded multiplication $R$-module $M$, $K$ is a graded $2$-absorbing primary $R$-submodule of $M$ $\Rightarrow$ $K$ is a graded weakly $2$-absorbing primary $R$-submodule of $M$ $\Rightarrow$ $K$ is a graded $\omega$-$2$-absorbing primary $R$-submodule of $M$ $\Rightarrow$ $K$ is a graded $(n+1)$-almost $2$-absorbing primary $R$-submodule of $M$ $\Rightarrow$ $K$ is a graded $n$-almost $2$-absorbing primary $R$-submodule of $M$ for all $n\geq2$ $\Rightarrow$ $K$ is a graded almost $2$-absorbing primary $R$-submodule of $M$.

\item Suppose that $Grad_{M}(K) = K$. Then $K$ is a graded $\phi$-$2$-absorbing primary $R$-submodule of $M$ if and only if $K$ is a graded $\phi$-$2$-absorbing $R$-submodule of $M$.

\item If $M$ is a graded multiplication $R$-module and $K$ is an idempotent $R$-submodule of $M$, then $K$ is a graded $\omega$-$2$-absorbing $R$-submodule of $M$, and $K$ is a graded $n$-almost $2$-absorbing $R$-submodule of $M$ for every $n\geq2$.

\item Let $M$ be a graded multiplication $R$-module. Then $K$ is a graded $n$-almost $2$-absorbing (resp. graded $n$-almost $2$-absorbing primary) $R$-submodule of $M$ for all $n\geq2$ if and only if $K$ is a graded $\omega$-$2$-absorbing (resp. graded $\omega$-$2$-absorbing primary) $R$-submodule of $M$.
\end{enumerate}
\end{thm}

\begin{proof}
\begin{enumerate}
\item It is straightforward.

\item Let $x, y\in h(R)$ and $m\in h(M)$ such that $xym\in K-\phi(K)$. Assume that $ym\notin Grad_{M}(K)$. Then $ym\notin K$ and then $x\in Grad((K :_{R} M))$ as $K$ is a graded $\phi$-primary $R$-submodule. Therefore, $xm\in Grad((K :_{R} M)M = Grad_{M}(K)$. Consequently, $K$ is graded $\phi$-$2$-absorbing primary.

\item It is clear by Remark \ref{Lemma 2.2} (4).

\item It is clear by Remark \ref{Lemma 2.2} (4).

\item The claim is obvious.

\item Since $K$ is an idempotent $R$-submodule, $K^{n}=K$ for all $n>0$, and then $\phi_{\omega}(K) =\displaystyle\bigcap_{n=1}^{\infty}K^{n}=K$. Thus $K$ is a graded $\omega$-$2$-absorbing $R$-submodule of $M$. By (3), we conclude that $K$ is a graded $n$-almost $2$-absorbing $R$-submodule of $M$ for all $n\geq2$.

\item Suppose that $K$ is a graded $n$-almost $2$-absorbing (resp. graded $n$-almost $2$-absorbing primary) $R$-submodule of $M$ for all $n\geq2$. Let $x, y\in h(R)$ and $m\in h(M)$ with $xym\in K$ but $xym\notin \displaystyle\bigcap_{n=1}^{\infty}K^{n}$. Hence $xym\notin K^{n}$ for some $n\geq2$. Since $K$ is graded $n$-almost $2$-absorbing (resp. graded $n$-almost $2$-absorbing primary) for all $n\geq2$, this implies either $xy\in (K :_{R} M)$ or $ym\in K$ or $xm\in K$ (resp. $xy\in (K :_{R} M)$ or $ym\in Grad_{M}(K)$ or $xm\in Grad_{M}(K)$). This completes the first implication. The converse is clear from (3) (resp. from (4)).
\end{enumerate}
\end{proof}

Let $M$ be a $G$-graded $R$-module and $K$ be a graded $R$-submodule of $M$. Then $M/K$ is $G$-graded by $(M/K)_{g}=(M_{g}+K)/K$ for all $g\in G$ (\cite{Nastasescue}).

\begin{lem}(\cite{Saber}, Lemma 3.2) Let $M$ be a graded $R$-module, $K$ be a graded $R$-submodule of $M$, and $N$ be an $R$-submodules of $M$ such that $K\subseteq N$. Then $N$ is a graded $R$-submodule of $M$ if and only if $N/K$ is a graded $R$-submodule of $M/K$.
\end{lem}

\begin{thm}\label{Theorem 2.4} Let $M$ be a graded $R$-module and $K$ be a proper graded $R$-submodule of $M$. Then the following hold:
\begin{enumerate}
\item $K$ is a graded $\phi$-$2$-absorbing $R$-submodule of $M$ if and only if $K/\phi(K)$ is a graded weakly $2$-absorbing $R$-submodule of $M/\phi(K)$.

\item $K$ is a graded $\phi$-$2$-absorbing primary $R$-submodule of $M$ if and only if $K/\phi(K)$ is a graded weakly $2$-absorbing primary $R$-submodule of $M/\phi(K)$.

\item $K$ is a graded $\phi$-prime $R$-submodule of $M$ if and only if $K/\phi(K)$ is a graded weakly prime $R$-submodule of $M/\phi(K)$.

\item $K$ is a graded $\phi$-primary $R$-submodule of $M$ if and only if $K/\phi(K)$ is a graded weakly primary $R$-submodule of $M/\phi(K)$.
\end{enumerate}
\end{thm}

\begin{proof}
\begin{enumerate}
\item If $\phi(K) = \emptyset$, then it is done. Suppose that $\phi(K)\neq\emptyset$. Let $x, y\in h(R)$ and $m+\phi(K)\in h(M/\phi(K))$ such that $\phi(K)\neq xy(m + \phi(K)) = xym + \phi(K)\in K/\phi(K)$. Then $m\in h(M)$ such that $xym\in K$, but $xym\notin \phi(K)$. Hence either $xy\in (K :_{R} M)$ or $ym\in K$ or $xm\in K$. So, $xy\in (K/\phi(K) :_{R} M/\phi(K))$ or $y(m + \phi(K))\in K/\phi(K)$ or $x(m + \phi(K))\in K/\phi(K)$, as desired. Conversely, let $x, y\in h(R)$ and $m\in h(M)$ such that $xym\in K$ and $xym\notin \phi(K)$. Then $m+\phi(K)\in h(M/\phi(K)$ such that $\phi(K)\neq xy(m + \phi(K))\in K/\phi(K)$. Hence $xy\in(K/\phi(K) :_{R} M/\phi(K))$ or $y(m + \phi(K))\in K/\phi(K)$ or $x(m + \phi(K))\in K/\phi(K)$. So, $xy\in(K :_{R} M)$ or $ym\in K$ or $xm\in K$. Thus $K$ is a graded $\phi$-$2$-absorbing $R$-submodule of $M$.

\item Let $x, y\in h(R)$ and $m+\phi(K)\in h(M/\phi(K))$ such that $\phi(K)\neq xy(m + \phi(K)) = xym + \phi(K)\in K/\phi(K)$. Then $m\in h(M)$ such that $xym\in K$, but $xym\notin \phi(K)$. Hence either $xy\in(K :_{R} M)$ or $ym\in Grad_{M}(K)$ or $xm\in Grad_{M}(K)$. So, $xy\in (K :_{R} M)$ or $y(m + \phi(K))\in Grad_{M}(K)/\phi(K)$ or $x(m + \phi(K))\in Grad_{M}(K)/\phi(K)$. The result holds since $Grad_{M}(K)/\phi(K) = Grad_{M/\phi(K)}(K/\phi(K))$. One can easily prove the converse.
\end{enumerate}
Similarly, one can easily prove (3) and (4).
\end{proof}

Let $M$ and $L$ be two $G$-graded $R$-modules. An $R$-homomorphism $f:M\rightarrow L$ is said to be a graded $R$-homomorphism if $f(M_{g})\subseteq L_{g}$ for all $g\in G$ (\cite{Nastasescue}).

\begin{lem}\label{2}(\cite{Dawwas Shashan}, Lemma 2.16) Suppose that $f:M\rightarrow L$ is a graded $R$-homomorphism. If $N$ is a graded $R$-submodule of $L$, then $f^{-1}(N)$ is a graded $R$-submodule of $M$.
\end{lem}

\begin{lem}\label{3}(\cite{Atani Saraei}, Lemma 4.8) Suppose that $f:M\rightarrow L$ is a graded $R$-homomorphism. If $K$ is a graded $R$-submodule of $M$, then $f(K)$ is a graded $R$-submodule of $f(M)$.
\end{lem}

\begin{thm}\label{Theorem 2.9} Suppose that $f : M\rightarrow L$ is a graded $R$-epimorphism. Let $\phi: GS(M)\rightarrow GS(M)\bigcup\{\emptyset\}$ and $\varphi: GS(L)\rightarrow GS(L)\bigcup\{\emptyset\}$ be functions. Then the following hold:
\begin{enumerate}
\item If $N$ is a graded $\varphi$-$2$-absorbing primary $R$-submodule of $L$ and $\phi(f^{-1}(N)) = f^{-1}(\varphi(N))$, then $f^{-1}(N)$ is a graded $\phi$-$2$-absorbing primary $R$-submodule of $M$.

\item If $K$ is a graded $\phi$-$2$-absorbing primary $R$-submodule of $M$ containing $Ker(f)$ and $\varphi(f(K)) =f(\phi(K))$, then $f(K)$ is a graded $\varphi$-$2$-absorbing primary $R$-submodule of $L$.

\item If $N$ is a graded $\varphi$-$2$-absorbing $R$-submodule of $L$ and $\phi(f^{-1}(N)) = f^{-1}(\varphi(N))$, then $f^{-1}(N)$ is a graded $\phi$-$2$-absorbing $R$-submodule of $M$.

\item If $K$ is a graded $\phi$-$2$-absorbing $R$-submodule of $M$ containing $Ker(f)$ and $\varphi(f(K)) =f(\phi(K))$, then $f(K)$ is a graded $\varphi$-$2$-absorbing $R$-submodule of $L$.
\end{enumerate}
\end{thm}

\begin{proof}
\begin{enumerate}
\item Since $f$ is epimorphism, $f^{-1}(N)$ is a proper graded $R$-submodule of $M$. Let $x, y\in h(R)$ and $m\in h(M)$ such that $xym\in f^{-1}(N)$ and $xym\notin f^{-1}(\varphi(N))$. Since $xym\in f^{-1}(N)$, $xyf(m)\in N$. Also, $\phi(f^{-1}(N)) = f^{-1}(\varphi(N))$ implies that $xyf(m)\notin \varphi(N)$. Thus $xyf(m)\in N-\varphi(N)$. Then $xy\in (N :_{R} L)$ or $xf(m)\in Grad_{L}(N)$ or $yf(m)\in Grad_{L}(N)$. Thus $xy\in (f^{-1}(N) :_{R} M)$ or $xm\in f^{-1}(Grad_{L}(N))$ or $ym\in f^{-1}(Grad_{L}(N))$. Since $f^{-1}(Grad_{L}(N))\subseteq Grad_{M}(f^{-1}(N))$, we conclude that $f^{-1}(N)$ is a graded $\phi$-$2$-absorbing primary $R$-submodule of $M$.

\item Let $x, y\in h(R)$ and $s\in h(L)$ such that $xys\in f(K)-\varphi(f(K))$. Since $f$ is graded epimorphism, there exists $m\in h(M)$ such that $s =f(m)$. Therefore, $f(xym)\in f(K)$ and so $xym\in K$ as $Ker(f)\subseteq K$. Since $\varphi(f(K)) = f(\phi(K))$, we have $xym\notin \phi(K)$. Hence $xym\in K-\phi(K)$. It implies that $xy\in (K :_{R} M)$ or $xm\in Grad_{M}(K)$ or $ym\in Grad_{M}(K)$. Thus $xy\in (f(K) :_{R}L)$ or $xs\in f(Grad_{M}(K))$ or $ys\in f(Grad_{M}(K))$. Since $Ker(f)\subseteq K$, $f(Grad_{M}(K)) =Grad_{L}(f(K))$, and then we are done.
\end{enumerate}
Similarly, one can easily prove (3) and (4).
\end{proof}

Let $M$ be a $G$-graded $R$-module and $S\subseteq h(R)$ be a multiplicative set. Then $S^{-1}M$ is a $G$-graded $S^{-1}R$-module with $(S^{-1}M)_{g}=\left\{\frac{m}{s},m\in M_{h}, s\in S\cap R_{hg^{-1}}\right\}$ for all $g\in G$, and $(S^{-1}R)_{g}=\left\{\frac{a}{s},a\in R_{h}, s\in S\cap R_{hg^{-1}}\right\}$ for all $g\in G$. If $K$ is a graded $R$-submodule of $M$, then $S^{-1}K$ is a graded $S^{-1}R$-submodule of $S^{-1}M$. Let $\phi: GS(M)\rightarrow GS(M)\bigcup\{\emptyset\}$ be a function and define $\phi_{S} : GS(S^{-1}M)\rightarrow GS(S^{-1}M)\bigcup\{\emptyset\}$ by $\phi_{S}(S^{-1}K) = S^{-1}\phi(K)$, and $\phi_{S}(S^{-1}K) =\emptyset$ when $\phi(K) =\emptyset$ for every graded $R$-submodule $K$ of $M$.

\begin{thm}\label{Theorem 2.11}Let $M$ be a graded $R$-module and $S\subseteq h(R)$ be a multiplicative set.
\begin{enumerate}
\item If $K$ is a graded $\phi$-$2$-absorbing primary $R$-submodule of $M$ and $S^{-1}K\neq S^{-1}M$, then $S^{-1}K$ is a graded $\phi_{S}$-$2$-absorbing primary $S^{-1}R$-submodule of $S^{-1}M$.

\item If $K$ is a graded $\phi$-$2$-absorbing $R$-submodule of $M$ and $S^{-1}K\neq S^{-1}M$, then $S^{-1}K$ is a graded $\phi_{S}$-$2$-absorbing $S^{-1}R$-submodule of $S^{-1}M$.
\end{enumerate}
\end{thm}

\begin{proof}
\begin{enumerate}
\item Let $x/s_{1}, y/s_{2}\in h(S^{-1}R)$ and $m/s\in h(S^{-1}M)$ such that $(x/s_{1})(y/s_{2})(m/s)\in S^{-1}K-\phi_{S}(S^{-1}K)$. Then $x, y\in h(R)$ and $m\in h(M)$ such that $uxym\in K-\phi(K)$ for some $u\in S$, and then $uxm\in Grad_{M}(K)$ or $uym\in Grad_{M}(K)$ or $xy\in (K:_{R}M)$. So, $(x/s_{1})(m/s)=(uxm)/(us_{1}s)\in S^{-1}(Grad_{M}(K))\subseteq Grad_{S^{-1}M}(S^{-1}K)$ or $(y/s_{2})(m/s)=(uym)/(us_{2}s)\in Grad_{S^{-1}M}(S^{-1}K)$ or $(x/s_{1})(y/s_{2})=(xy)/(s_{1}s_{2})\in S^{-1}(K :_{R} M)\subseteq(S^{-1}K :_{S^{-1}R} S^{-1}M)$.
\end{enumerate}
Similarly, one can easily prove (2).
\end{proof}

Let $M_{1}$ be a $G$-graded $R_{1}$-module, $M_{2}$ be a $G$-graded $R_{2}$-module and $R = R_{1}\times R_{2}$. Then $M = M_{1}\times M_{2}$
is $G$-graded $R$-module with $M_{g} = (M_{1})_{g}\times(M_{2})_{g}$ for all $g\in G$, where $R_{g} = (R_{1})_{g}\times(R_{2})_{g}$ for all $g\in G$, (\cite{Nastasescue}).

\begin{lem}(\cite{Saber}, Lemma 3.12) Let $M_{1}$ be a $G$-graded $R_{1}$-module, $M_{2}$ be a $G$-graded $R_{2}$-module, $R = R_{1}\times R_{2}$ and $M = M_{1}\times M_{2}$. Then $L = N\times K$ is a graded $R$-submodule of $M$ if and only if $N$ is a graded $R_{1}$-submodule of $M_{1}$ and $K$ is a graded $R_{2}$-submodule of $M_{2}$.
\end{lem}

\begin{lem}\label{4} Let $M_{1}$ be a $G$-graded $R_{1}$-module, $M_{2}$ be a $G$-graded $R_{2}$-module, $R = R_{1}\times R_{2}$ and $M = M_{1}\times M_{2}$. Suppose that $\varphi_{1}:GS(M_{1})\rightarrow GS(M_{1})\bigcup\{\emptyset\}$, $\varphi_{2}:GS(M_{2})\rightarrow GS(M_{2})\bigcup\{\emptyset\}$ be functions and $\phi=\varphi_{1}\times\varphi_{2}$. Assume that $K = K_{1}\times M_{2}$ for some proper graded $R_{1}$-submodule $K_{1}$ of $M_{1}$. If $K$ is a graded $\phi$-$2$-absorbing $R$-submodule of $M$, then $K_{1}$ is a graded $\varphi_{1}$-$2$-absorbing $R_{1}$-submodule of $M_{1}$.
\end{lem}

\begin{proof} Let $x_{1}, y_{1}\in h(R_{1})$ and $m_{1}\in h(M_{1})$ such that $x_{1}y_{1}m_{1}\in K_{1}-\varphi_{1}(K_{1})$. Then for any $m\in h(M_{2})$, $(x_{1}, 1), (y_{1}, 1)\in h(R)$ and $(m_{1}, m)\in h(M)$ such that $(x_{1}, 1)(y_{1}, 1)(m_{1},m)\in (K_{1}\times M_{2})-(\varphi_{1}(K_{1})\times \varphi_{2}(M_{2})) = K-\phi(K)$. Since $K$ is a graded $\phi$-$2$-absorbing $R$-submodule of $M$, we get either $(x_{1}, 1)(y_{1}, 1)\in ((K_{1}\times M_{2}) :_{R} M_{1}\times M_{2})$ or $(x_{1}, 1)(m_{1}, m)\in (K_{1}\times M_{2})$ or $(y_{1}, 1)(m_{1}, m)\in (K_{1}\times M_{2})$. So clearly, we conclude that $x_{1}y_{1}\in (K_{1} :_{R_{1}} M_{1})$ or $x_{1}m_{1}\in K_{1}$ or $y_{1}m_{1}\in K_{1}$. Therefore, $K_{1}$ is a graded $\varphi_{1}$-$2$-absorbing $R_{1}$-submodule of $M_{1}$.
\end{proof}

\begin{thm}\label{Theorem 2.16}Let $M_{1}$ be a $G$-graded $R_{1}$-module, $M_{2}$ be a $G$-graded $R_{2}$-module, $R = R_{1}\times R_{2}$ and $M = M_{1}\times M_{2}$. Suppose that $\varphi_{1}:GS(M_{1})\rightarrow GS(M_{1})\bigcup\{\emptyset\}$, $\varphi_{2}:GS(M_{2})\rightarrow GS(M_{2})\bigcup\{\emptyset\}$ be functions and $\phi=\varphi_{1}\times\varphi_{2}$. Assume that $K = K_{1}\times M_{2}$ for some proper graded $R_{1}$-submodule $K_{1}$ of $M_{1}$. Then the following hold:
\begin{enumerate}
\item If $\varphi_{2}(M_{2}) = M_{2}$, then $K$ is a graded $\phi$-$2$-absorbing $R$-submodule of $M$ if and only if $K_{1}$ is a graded $\varphi_{1}$-$2$-absorbing $R_{1}$-submodule of $M_{1}$.

\item If $\varphi_{2}(M_{2})\neq M_{2}$, then $K$ is a graded $\phi$-$2$-absorbing $R$-submodule of $M$ if and only if $K_{1}$ is a graded $2$-absorbing $R_{1}$-submodule of $M_{1}$.
\end{enumerate}
\end{thm}

\begin{proof}
\begin{enumerate}
\item Suppose that $K_{1}$ is a graded $\varphi_{1}$-$2$-absorbing $R_{1}$-submodule of $M_{1}$. Let $x=(x_{1}, x_{2}), y=(y_{1}, y_{2})\in h(R)$ and $m=(m_{1}, m_{2})\in h(M)$ such that $xym\in K-\phi(K)$. Since $\varphi_{2}(M_{2}) = M_{2}$, we get that $x_{1}, y_{1}\in h(R_{1})$ and $m_{1}\in h(M_{1})$ such that $x_{1}y_{1}m_{1}\in K_{1}-\varphi_{1}(K_{1})$, and this implies that either $x_{1}y_{1}\in (K_{1} :_{R_{1}}M_{1})$ or $x_{1}m_{1}\in K_{1}$ or $y_{1}m_{1}\in K_{1}$. Thus either $xy\in (K :_{R} M)$ or $xm\in K$ or $ym\in K$. Hence, $K$ is a graded $\phi$-$2$-absorbing $R$-submodule of $M$. The converse holds from Lemma \ref{4}.

\item Suppose that $K$ is a graded $\phi$-$2$-absorbing $R$-submodule of $M$. Since $\varphi_{2}(M_{2})\neq M_{2}$, there exists $m_{2}\in M_{2}-\varphi_{2}(M_{2})$ and then there exists $g\in G$ such that $(m_{2})_{g}\in M_{2}-\varphi_{2}(M_{2})$. Assume that $K_{1}$ is not a graded $2$-absorbing $R_{1}$-submodule of $M_{1}$. By Lemma \ref{4}, $K_{1}$ is a graded $\varphi_{1}$-$2$-absorbing $R_{1}$-submodule of $M_{1}$. Hence, there exist $x_{1}, y_{1}\in h(R_{1})$ and $m_{1}\in h(M_{1})$ such that $x_{1}y_{1}m_{1}\in \varphi_{1}(K_{1})$, $x_{1}y_{1}\notin (K_{1}:_{R_{1}}M_{1})$, $x_{1}m_{1}\notin K_{1}$ and $y_{1}m_{1}\notin K_{1}$. So, $(x_{1}, 1)(y_{1}, 1)(m_{1},(m_{2})_{g})\in (K_{1}\times M_{2})-(\varphi_{1}(K_{1})\times \varphi_{2}(M_{2})) = (K_{1}\times M_{2})-\phi(K_{1}\times M_{2})$ which implies that $x_{1}y_{1}\in (K_{1} :_{R_{1}} M_{1})$ or $x_{1}m_{1}\in K_{1}$ or $y_{1}m_{1}\in K_{1}$, which is a contradiction. Thus $K_{1}$ is a graded $2$-absorbing $R_{1}$-submodule of $M_{1}$. Conversely, if $K_{1}$ is a graded $2$-absorbing $R_{1}$-submodule of $M_{1}$, then $K = K_{1}\times M_{2}$ is a graded $2$-absorbing $R$-submodule of $M$ by (\cite{Zoubi Azaizeh}, Theorem 3.3). Hence $K$ is a graded $\phi$-$2$-absorbing $R$-submodule of $M$ for any $\phi$.
\end{enumerate}
\end{proof}

\begin{lem} Let $M_{1}$ be a $G$-graded $R_{1}$-module, $M_{2}$ be a $G$-graded $R_{2}$-module, $R = R_{1}\times R_{2}$ and $M = M_{1}\times M_{2}$. Suppose that $\varphi_{1}:GS(M_{1})\rightarrow GS(M_{1})\bigcup\{\emptyset\}$, $\varphi_{2}:GS(M_{2})\rightarrow GS(M_{2})\bigcup\{\emptyset\}$ be functions and $\phi=\varphi_{1}\times\varphi_{2}$. Assume that $K = K_{1}\times M_{2}$ for some proper graded $R_{1}$-submodule $K_{1}$ of $M_{1}$. If $K$ is a graded $\phi$-$2$-absorbing primary $R$-submodule of $M$, then $K_{1}$ is a graded $\varphi_{1}$-$2$-absorbing primary $R_{1}$-submodule of $M_{1}$.
\end{lem}

\begin{proof} It can be easily proved by using a similar argument in the proof of Lemma \ref{4}.
\end{proof}

\begin{thm}\label{Theorem 2.17}Let $M_{1}$ be a $G$-graded $R_{1}$-module, $M_{2}$ be a $G$-graded $R_{2}$-module, $R = R_{1}\times R_{2}$ and $M = M_{1}\times M_{2}$. Suppose that $\varphi_{1}:GS(M_{1})\rightarrow GS(M_{1})\bigcup\{\emptyset\}$, $\varphi_{2}:GS(M_{2})\rightarrow GS(M_{2})\bigcup\{\emptyset\}$ be functions and $\phi=\varphi_{1}\times\varphi_{2}$. Assume that $K = K_{1}\times M_{2}$ for some proper graded $R_{1}$-submodule $K_{1}$ of $M_{1}$. Then the following hold:
\begin{enumerate}
\item If $\varphi_{2}(M_{2}) = M_{2}$, then $K$ is a graded $\phi$-$2$-absorbing primary $R$-submodule of $M$ if and only if $K_{1}$ is a graded $\varphi_{1}$-$2$-absorbing primary $R_{1}$-submodule of $M_{1}$.

\item If $\varphi_{2}(M_{2})\neq M_{2}$, then $K$ is a graded $\phi$-$2$-absorbing primary $R$-submodule of $M$ if and only if $K_{1}$ is a graded $2$-absorbing primary $R_{1}$-submodule of $M_{1}$.
\end{enumerate}
\end{thm}

\begin{proof}
\begin{enumerate}
\item It can be easily proved by using a similar argument in the proof of Theorem \ref{Theorem 2.16} (1).

\item Suppose that $K_{1}$ is a graded $2$-absorbing primary $R_{1}$-submodule of $M_{1}$. Then $K = K_{1}\times M_{2}$ is a graded $2$-absorbing primary $R$-submodule of $M$ by (\cite{Ece Yetkin Celikel}, Theorem 18). Hence $K$ is a graded $\phi$-$2$-absorbing $R$-submodule of $M$ for any $\phi$. The remaining of this proof is similar to Theorem \ref{Theorem 2.16} (2).
\end{enumerate}
\end{proof}

\begin{thm}\label{Theorem 2.20} Let $M$ be a $G$-graded $R$-module $g\in G$ and $K$ be a $g$-$\phi$-primary $R$-submodule of $M$. Suppose that $x\in R_{e}$ and $m\in M_{g}$ such that $xm\in \phi(K)$, $x\notin Grad((K:_{R}M))$ and $m\notin K$. Then
\begin{enumerate}
\item $xK_{g}\subseteq \phi(K)$.

\item $(K :_{R_{e}} M)m\subseteq\phi(K)$.

\item $(K :_{R_{e}} M)K_{g}\subseteq\phi(K)$.
\end{enumerate}
\end{thm}

\begin{proof}
\begin{enumerate}
\item Suppose that $xK_{g}\nsubseteq\phi(K)$. Then there exists $k\in K_{g}$ such that $xk\notin \phi(K)$, and then $x(m + k)\notin \phi(K)$. Since $x(m + k)\in K$ and $x\notin Grad((K :_{R} M))$, we deduce that $m + k\in K$ as $K$ is a $g$-$\phi$-primary $R$-submodule of $M$. So $m\in K$, which is a contradiction. So, $xK_{g}\subseteq\phi(K)$.

\item Suppose that $(K :_{R_{e}} M)m\nsubseteq\phi(K)$. Then there exists $y\in (K:_{R_{e}}M)$ such that $ym\notin \phi(K)$, and then $(x + y)m\notin \phi(K)$ as $xm\in \phi(K)$. Since $ym\in K$, we get $(x + y)m\in K$. Since $m\notin K$, we have that $x + y\in Grad((K :_{R} M))$. Hence
    $x\in Grad((K :_{R} M))$, which is a contradiction.

\item Suppose that $(K :_{R_{e}} M)K_{g}\nsubseteq\phi(K)$. Then there exist $y\in (K :_{R_{e}} M)$ and $k\in K_{g}$ such that $yk\notin \phi(K)$. By (1) and (2), $(x+y)(m+k)\in K-\phi (K)$. So, either $x+y\in Grad((K :_{R} M))$ or $m + k\in K$. Thus we have either $x\in Grad((K :_{R} M))$ or $m\in K$, which is a contradiction.
\end{enumerate}
\end{proof}

\begin{rem}Note that if $K$ is a $g$-$\phi$-primary $R$-submodule of $M$ which is not $g$-primary, then there exist $x\in R_{e}$ and $m\in M_{g}$ such that $xm\in \phi(K)$, $x\notin Grad((K:_{R}M))$ and $m\notin K$. So, every $g$-$\phi$-primary $R$-submodule, which is not $g$-primary, satisfies Theorem \ref{Theorem 2.20}.
\end{rem}

\begin{thm}\label{Theorem 2.24} Let $M$ be a $G$-graded $R$-module, $g\in G$ and $K$ be a $g$-$\phi$-$2$-absorbing $R$-submodule of $M$. Suppose that $x, y\in R_{e}$ and $m\in M_{g}$ such that $xym\in \phi(K)$, $xy\notin (K:_{R}M)$, $xm\notin K$ and $ym\notin K$. Then
\begin{enumerate}
\item $xyK_{g}\subseteq\phi(K)$.

\item $x(K :_{R_{e}} M)m\subseteq\phi(K)$.

\item $y(K :_{R_{e}} M)m\subseteq\phi(K)$.

\item $(K :_{R_{e}} M)^{2}m\subseteq\phi(K)$.
\end{enumerate}
\end{thm}

\begin{proof}
\begin{enumerate}
\item Suppose that $xyK_{g}\nsubseteq\phi(K)$. Then there exists $k\in K_{g}$ with $xyk\notin\phi(K)$, and then $xy(m+k)\notin\phi(K)$. Since $xy(m + k) = xym + xyk\in K$ and $xy\notin(K :_{R} M)$, we conclude that $x(m + k)\in K$ or $y(m + k)\in K$. So, $xm\in K$ or $ym\in K$, which is a contradiction. Thus $xyK\subseteq\phi(K)$.
    
\item Suppose that $x(K :_{R_{e}} M)m\nsubseteq\phi(K)$. Then there exists $a\in (K:_{R_{e}}M)$ such that $xam\notin \phi(K)$, and then $x(y + a)m\notin \phi(K)$ as $xym\in \phi(K)$. Since $am\in K$, we obtain that $x(y + a)m\in K$. Then $xm\in K$ or $(y + a)m\in K$ or $x(y + a)\in (K :_{R} M)$.
Hence $xm\in K$ or $ym\in K$ or $xy\in (K :_{R} M)$, which is a contradiction. Hence, $x(K :_{R_{e}} M)m\subseteq\phi(K)$.

\item It can be easily proved by using a similar argument in the proof of part (2).

\item Assume that $(K :_{R_{e}} M)^{2}m\nsubseteq\phi(K)$. Then there exist $a, b\in (K :_{R_{e}} M)$ such that $abm\notin \phi(K)$, and then by parts (2) and (3), $(x + a)(y + b)m\notin \phi(K)$. Clearly, $(x + a)(y + b)m\in K$. Then $(x + a)m\in K$ or $(y + b)m\in K$ or $(x + a)(y + b)\in (K :_{R} M)$. Therefore, $xm\in K$ or $ym\in K$ or $xy\in (K :_{R} M)$, which is a contradiction. Consequently, $(K :_{R_{e}} M)^{2}m\subseteq\phi(K)$.
\end{enumerate}
\end{proof}

\begin{rem}Note that if $K$ is a $g$-$\phi$-$2$-absorbing $R$-submodule of $M$ which is not $g$-$2$-absorbing, then there exist $x, y\in R_{e}$ and $m\in M_{g}$ such that $xym\in \phi(K)$, $xy\notin (K:_{R}M)$, $xm\notin K$ and $ym\notin K$. So, every $g$-$\phi$-$2$-absorbing $R$-submodule, which is not $g$-$2$-absorbing, satisfies Theorem \ref{Theorem 2.24}.
\end{rem}

\begin{thm}\label{Theorem 2.24 (1)} Let $M$ be a $G$-graded $R$-module, $g\in G$ and $K$ be a $g$-$\phi$-$2$-absorbing primary $R$-submodule of $M$. Suppose that $x, y\in R_{e}$ and $m\in M_{g}$ such that $xym\in \phi(K)$, $xy\notin (K:_{R}M)$, $xm\notin Grad_{M}(K)$ and $ym\notin Grad_{M}(K)$. Then
\begin{enumerate}
\item $xyK_{g}\subseteq\phi(K)$.

\item $x(K :_{R_{e}} M)m\subseteq\phi(K)$.

\item $y(K :_{R_{e}} M)m\subseteq\phi(K)$.

\item $(K :_{R_{e}} M)^{2}m\subseteq\phi(K)$.
\end{enumerate}
\end{thm}

\begin{proof}It can be easily proved by using a similar argument in the proof of Theorem \ref{Theorem 2.24}.
\end{proof}

\begin{rem}Note that if $K$ is a $g$-$\phi$-$2$-absorbing primary $R$-submodule of $M$ which is not $g$-$2$-absorbing primary, then there exist $x, y\in R_{e}$ and $m\in M_{g}$ such that $xym\in \phi(K)$, $xy\notin (K:_{R}M)$, $xm\notin Grad_{M}(K)$ and $ym\notin Grad_{M}(K)$. So, every $g$-$\phi$-$2$-absorbing primary $R$-submodule, which is not $g$-$2$-absorbing primary, satisfies Theorem \ref{Theorem 2.24 (1)}.
\end{rem}

\begin{thm}\label{Theorem 2.25} Let $M$ be a $G$-graded $R$-module and $g\in G$. If $K$ is a $g$-$\phi$-$2$-absorbing primary $R$-submodule of $M$ that is not $g$-$2$-absorbing primary, then $(K :_{R_{e}} M)^{2}K_{g}\subseteq\phi(K)$.
\end{thm}

\begin{proof}Since $K$ is a $g$-$\phi$-$2$-absorbing primary $R$-submodule of $M$ that is not $g$-$2$-absorbing primary, there exist $x, y\in R_{e}$ and $m\in M_{g}$ such that $xym\in \phi(K)$, $xy\notin (K:_{R}M)$, $xm\notin Grad_{M}(K)$ and $ym\notin Grad_{M}(K)$. Suppose that $(K :_{R_{e}} M)^{2}K_{g}\nsubseteq\phi(K)$. Then there are $a, b\in (K :_{R_{e}} M)$ and $k\in K_{g}$ such that $abk\notin \phi(K)$. By Theorem \ref{Theorem 2.24 (1)}, we get $(x + a)(y + b)(m +k)\in K-\phi(K)$. So, $(x + a)(m + k)\in Grad_{M}(K)$ or $(y + b)(m + k)\in Grad_{M}(K)$ or $(x + a)(y + b)\in (K :_{R} M)$. Therefore, $xm\in Grad_{M}(K)$ or $ym\in Grad_{M}(K)$ or $xy\in (K :_{R} M)$, which is a contradiction. Hence, $(K :_{R_{e}} M)^{2}K_{g}\subseteq\phi(K)$.
\end{proof}

\begin{prop}\label{Proposition 2.34}Let $M$ be a graded $R$-module and $K$ be a graded $\phi$-$2$-absorbing primary $R$-submodule of $M$. If $\phi(K)$ is a graded $2$-absorbing primary $R$-submodule of $M$, then $K$ is a graded-$2$-absorbing primary $R$-submodule of $M$.
\end{prop}

\begin{proof}Assume that $x, y\in h(R)$ and $m\in h(M)$ such that $xym\in K$. If $xym\in \phi(K)$, then we conclude that $xm\in Grad_{M}(\phi(K))\subseteq Grad_{M}(K)$ or $ym\in Grad_{M}(\phi(K))\subseteq Grad_{M}(K)$ or $xy\in (\phi(K) :_{R} M)\subseteq(K :_{R} M)$ since $\phi(K)$ is graded $2$-absorbing primary, and so the result holds. If $xym\notin \phi(K)$, then the result holds easily since $K$ is graded $\phi$-$2$-absorbing primary.
\end{proof}

\begin{thm}\label{Theorem 2.39} Let $M$ be a graded $R$-module and $K$ be a graded $R$-submodule of $M$ with $\phi(Grad_{M}(K))\subseteq\phi(K)$. If $Grad_{M}(K)$ is a graded $\phi$-prime $R$-submodule of $M$, then $K$ is a graded $\phi$-$2$-absorbing primary $R$-submodule of $M$.
\end{thm}

\begin{proof}Let $x,y\in h(R)$ and $m\in h(M)$ such that $xym\in K-\phi(K)$ and $xm\notin Grad_{M}(K)$. Since $Grad_{M}(K)$ is a graded $\phi$-prime $R$-submodule and $xym\in Grad_{M}(K)-\phi(Grad_{M}(K))$, $y\in (Grad_{M}(K) :_{R} M)$. So, $ym\in Grad_{M}(K)$. Consequently, $K$ is a graded $\phi$-$2$-absorbing primary $R$-submodule of $M$.
\end{proof}

\section*{\textbf{Acknowledgement}}

This research was funded by the Deanship of Scientific Research at Princess Nourah bint Abdulrahman University through the Fast-track Research Funding Program.

\end{document}